\documentclass[12pt,a4paper]{amsart}
\usepackage{amssymb,eucal}

\calclayout

\newcommand{\+}{\nobreakdash-}
\renewcommand{\:}{\colon}

\newcommand{\rarrow}{\longrightarrow}

\newcommand{\ot}{\otimes}
\newcommand{\ocn}{\odot}

\DeclareMathOperator{\Hom}{Hom}
\DeclareMathOperator{\Ext}{Ext}

\DeclareMathOperator{\Cohom}{Cohom}

\newcommand{\modl}{{\operatorname{\mathsf{--mod}}}}
\newcommand{\Modl}{{\operatorname{\mathsf{--Mod}}}}

\newcommand{\Contra}{{\operatorname{\mathsf{--Contra}}}}

\newcommand{\Comodl}{{\operatorname{\mathsf{--Comod}}}}
\newcommand{\Comodr}{{\operatorname{\mathsf{Comod--}}}}
\newcommand{\comodl}{{\operatorname{\mathsf{--comod}}}}

\newcommand{\Vect}{{\operatorname{\mathsf{--Vect}}}}

\newcommand{\inj}{{\mathsf{inj}}}
\newcommand{\proj}{{\mathsf{proj}}}

\newcommand{\rop}{{\mathrm{op}}}
\newcommand{\sop}{{\mathsf{op}}}

\newcommand{\lrarrow}{\mskip.5\thinmuskip\relbar\joinrel\relbar
   \joinrel\rightarrow\mskip.5\thinmuskip\relax}
\newcommand{\llarrow}{\mskip.5\thinmuskip\leftarrow\joinrel\relbar
   \joinrel\relbar\mskip.5\thinmuskip}

\newcommand{\bu}{{\text{\smaller\smaller$\scriptstyle\bullet$}}}
\newcommand{\oc}{\mathbin{\text{\smaller$\square$}}}

\newcommand{\R}{\widehat R}

\newcommand{\fP}{\mathfrak P}
\newcommand{\fQ}{\mathfrak Q}
\newcommand{\cM}{\mathcal M}
\newcommand{\cN}{\mathcal N}
\newcommand{\C}{\mathcal C}
\newcommand{\J}{\mathcal J}

\newcommand{\add}{\mathsf{add}}
\newcommand{\Sym}{\mathcal S\mathit{ym}}
\newcommand{\Ten}{\mathcal T\!\mathit{en}}

\newcommand{\sA}{\mathsf A}
\newcommand{\sB}{\mathsf B}

\newcommand{\sK}{\mathsf K}

\newcommand{\boZ}{\mathbb Z}
\newcommand{\boQ}{\mathbb Q}

\newcommand{\Section}[1]{\bigskip\section{#1}\medskip}
\theoremstyle{plain}
\newtheorem{thm}{Theorem}[section]
\newtheorem{prop}[thm]{Proposition}
\newtheorem{lem}[thm]{Lemma}
\newtheorem{cor}[thm]{Corollary}
\newtheorem{conc}[thm]{Conclusion}
\theoremstyle{definition}
\newtheorem{rem}[thm]{Remark}
\newtheorem{ex}[thm]{Example}

\begin{document}

\title{A bounded below, noncontractible, \\
acyclic complex of projective modules}

\author{Leonid Positselski}

\address{Institute of Mathematics, Czech Academy of Sciences \\
\v Zitn\'a~25, 115~67 Prague~1 \\ Czech Republic} 

\email{positselski@math.cas.cz}

\begin{abstract}
 We construct examples of bounded below, noncontractible, acyclic
complexes of finitely generated projective modules over some rings $S$,
as well as bounded above, noncontractible, acyclic complexes of
injective modules.
 The rings $S$ are certain rings of infinite matrices with entries in
the rings of commutative polynomials or formal power series in
infinitely many variables.
 In the world of comodules or contramodules over coalgebras over
fields, similar examples exist over the cocommutative symmetric
coalgebra of an infinite-dimensional vector space.
 A simpler, universal example of a bounded below, noncontractible,
acyclic complex of free modules with one generator, communicated to
the author by Canonaco, is included at the end of the paper.
\end{abstract}

\maketitle

\tableofcontents

\section*{Introduction}
\medskip

 Bounded above acyclic complexes of projective objects are contractible.
 So are bounded below acyclic complexes of injective objects.
 On the other hand, there is an easy, thematic example of a doubly
unbounded, acyclic, noncontractible complex of finitely generated
projective-injective modules over the algebra of dual numbers
$R=k[\epsilon]/(\epsilon^2)$ (over any field~$k$):
\begin{equation} \label{unbounded-over-dual-numbers}
 \dotsb\lrarrow R\overset{\epsilon*}\lrarrow R\overset{\epsilon*}
 \lrarrow R\lrarrow\dotsb
\end{equation}
 We refer to~\cite[Prologue]{Prel}, \cite[Sections~7.4\+-7.5]{Pksurv}
and the references therein for a discussion of the role of
the complex~\eqref{unbounded-over-dual-numbers} in the context of
derived Koszul duality and derived categories of the second kind.

 Do there exist bounded below, noncontractible, acyclic complexes of
projective modules; and if so, under what rings?
 Dual-analogously, are there any bounded above, noncontractible, acyclic
complexes of injective modules?
 These questions were posed, in the context of potential applications
to the Finitistic Dimension Conjecture, in the recent preprint of
Shaul~\cite{Sha}.
 According to~\cite[Theorem~5.1]{Sha}, nonexistence of such complexes of
projective/injective modules over a two-sided Noetherian ring $S$ with
a dualizing complex would imply finiteness of the finitistic dimensions
of~$S$.

 The aim of the present paper is to show that, over certain rather big
rings $S$, such complexes do exist.
 The examples of rings $S$ which we obtain are certainly noncommutative
and non-Noetherian.
 The more explicit ones among them are rings of column-finite or
row/column-zero-convergent infinite matrices with entries in the rings
of commutative polynomials or formal power series in infinitely
many variables.

 On the other hand, in the world of coalgebras over fields, we
demonstrate examples of bounded above, acyclic, noncontractible
complexes of injective comodules and bounded below, acyclic,
noncontractible complexes of projective contramodules over certain
\emph{cocommutative} coalgebras dual to algebras of formal power series
in infinitely many variables.
 These examples go back to~\cite[Section~0.2.7]{Psemi}, where they were
very briefly discussed in the context of semi-infinite homological
algebra and derived comodule-contramodule correspondence.

 Almost all the examples presented in this paper are based on one idea,
namely, that of the dual Koszul complex of the ring of polynomials in
infinitely many variables.
 A straightforward realization of this idea is possible in the worlds
of comodules and contramodules, but we need an additional trick with
a passage to infinite matrices in order to produce examples
of complexes of projective/injective modules.
 The only exception is the (much simpler) \emph{universal} example,
communicated to the author by A.~Canonaco.
 We reproduce it at the end of the paper in
Example~\ref{universal-example}.

 The approach to the Finitistic Dimension Conjecture developed
in~\cite{Sha0,Sha} goes back to Rickard's paper~\cite{Rick}, where it
was shown that if the injective modules over a finite-dimensional
algebra generate its unbounded derived category as a triangulated
category with coproducts, then the finitistic dimension is finite.
 A counterexample in~\cite[Theorem~3.5]{Rick} shows that for
the ring of commutative polynomials in infinitely many variables,
the generation property fails.
 Our examples in this paper follow in the footsteps
of~\cite[Section~0.2.7]{Psemi} and~\cite[Theorem~3.5]{Rick}.
 We also provide some details of the claims
in~\cite[Section~0.2.7]{Psemi} which were skipped in
the book~\cite{Psemi}.

\subsection*{Acknowledgement}
 This paper was inspired by Liran Shaul's talk at the Algebra seminar
in Prague, organized by Jan Trlifaj.
 I~want to thank both the speaker and the organizer of the seminar.
 I~also wish to thank Alberto Canonaco for communicating his example
to me and giving a kind permission to reproduce it here
(see Example~\ref{universal-example}).
 The author is supported by the GA\v CR project 23-05148S and
the Czech Academy of Sciences (RVO~67985840).

\Section{Projective, Flat, and Injective Bounded Acyclicity
Problems}

 The general convention in this paper is that complexes are presumed
to be cohomologically graded, so the differential raises the degree.
 A complex $C^\bu=(C^n,\>d_n\:C^n\to C^{n+1})$ is called
\emph{bounded above} if $C^n=0$ for $n\gg0$, and $C^\bu$ is
\emph{bounded below} if $C^n=0$ for $n\ll0$.
 In this notation, it is a standard fact that every bounded above
acyclic complex of projective modules/objects (in an abelian or
exact category) is contractible, and every bounded below acyclic
complex of injective modules/objects is contractible.
 When we occasionally consider homologically graded complexes, we
use the notation with lower indices,
$P_\bu=(P_n,\>d_n\:P_n\to P_{n-1})$.

 Let $S$ be an associative ring.
 The two ``wrong-sided bounded projective/injective acyclicity
problems'' posed in~\cite[Theorem~5.1(4\+-5)]{Sha} are:
\begin{itemize}
\item Is every bounded above acyclic complex of
injective $S$\+modules contractible?
\item Is every bounded below acyclic complex of
projective $S$\+modules contractible?
\end{itemize}

 In addition to the above two, we would like to ask a similar
question about flat $S$\+modules.
 Here one has to be careful: even a two-sided bounded acyclic complex
of flat modules need not be contractible.
 However, such a complex is always \emph{pure acyclic}, or in other
words, has flat modules of cocycles.
 Thus we ask:
\begin{itemize}
\item Is every bounded below acyclic complex of
flat $S$\+modules pure acyclic?
\end{itemize} 

 Given a ring $S$ and a left $S$\+module $M$, the \emph{character
module} $M^+=\Hom_\boZ(M,\boQ/\boZ)$ is a right $S$\+module.
 The following lemma is well-known.

\begin{lem} \label{flat-injective-character}
 A left $S$\+module $F$ is flat if and only if the right $S$\+module
$F^+$ is injective.  \qed
\end{lem}

 The next proposition explains the connection between
the injective, flat, and projective wrong-sided bounded acyclicity
questions, and shows that presenting a counterexample to
the ``projective'' question is enough.

\begin{prop} \label{projective-is-enough-to-refute}
 Given a ring $S$, consider the following three properties:
\begin{enumerate}
\item Every bounded above acyclic complex of injective right
$S$\+modules is contractible.
\item Every bounded below acyclic complex of flat left
$S$\+modules is pure acyclic.
\item Every bounded below acyclic complex of projective left
$S$\+modules is contractible.
\end{enumerate}
 Then the implications
\textup{(1)\,$\Longrightarrow$\,(2)\,$\Longrightarrow$\,(3)} hold.
\end{prop}

\begin{proof}
 (1)\,$\Longrightarrow$\,(2)
 Let $F^\bu=(0\to F^0\to F^1\to F^2\to\dotsb)$ be a bounded below
acyclic complex of flat left $S$\+modules.
 Then, by  the direct implication of
Lemma~\ref{flat-injective-character},
$F^{\bu,+}=(\dotsb\to F^{2,+}\to F^{1,+}\to F^{0,+}\to0)$ is
a bounded above acyclic complex of injective right $S$\+modules.
 A complex of injective modules is contractible if and only if its
modules of cocycles are injective.
 If this is the case for the complex $F^{\bu,+}$, then the inverse
implication of Lemma~\ref{flat-injective-character} tells that
the modules of cocycles of the complex $F^\bu$ are flat; so $F^\bu$
is a pure acyclic complex of flat modules.

 (2)\,$\Longrightarrow$\,(3)
 By Neeman's
theorem~\cite[Theorem~8.6\,(iii)\,$\Rightarrow$\,(i)]{Neem},
any pure acyclic complex of projective modules is contractible.
 (Cf.~\cite[proof of Theorem~A.7]{Sha}.)
\end{proof}

\Section{The Injective Construction of Acyclic Complex of Projectives}

 Let $k$~be a field, $(x_\alpha)_{\alpha\in A}$ be an infinite set of
variables, and $R=k[x_\alpha:\alpha\in A]$ be the commutative ring
of polynomials in the variables~$x_\alpha$ over~$k$.
 Endow the one-dimensional vector space $k$ over~$k$ with
the $R$\+module structure by the obvious rule: all the elements
$x_\alpha\in R$ act by zero in~$k$.

\begin{thm}[Rickard] \label{rickard-theorem}
 For any injective $R$\+module $J$ and all integers $n\ge0$, one has\/
$\Ext_R^n(J,k)=0$.
\end{thm}

\begin{proof}
 For a countably infinite set of variables~$x_\alpha$, this is
formulated and proved in~\cite[Theorem~3.5]{Rick}.
 The general case of a possibly uncountable index set $A$ is similar.
 One represents $A$ as the union of its finite subsets $B\subset A$,
so the ring $R$ the direct limit of the related polynomial rings
$R_B$ in finitely many variables, considers the direct limit of
finite Koszul complex indexed by the finite subsets $B\subset A$, etc.
 (Cf.\ the proof of Theorem~\ref{dual-rickard} below for some further
details.)
\end{proof}

 Let $\sA$ be an additive category and $M\in\sA$ be an object.
 Then we denote by $\add(M)$ the full subcategory in $\sA$ formed by
the direct summands of finite direct sums of copies of~$M$.
 The following lemma is a straightforward category-theoretic
generalization of a well-known module-theoretic observation going back
to Dress~\cite{Dr}.

\begin{lem} \label{dress-lemma}
 Let\/ $\sA$ be an idempotent-complete additive category and
$M\in\sA$ be an object. \par
\textup{(a)} Let $S=\Hom_\sA(M,M)^\rop$ be the opposite ring to
the endomorphism ring of the object $M\in\sA$; so the ring $S$
acts on the object $M$ on the right.
 Then the covariant functor\/ $\Hom_\sA(M,{-})\:\sA\rarrow S\Modl$
restricts to an equivalence of additive categories
$$
 \Hom_\sA(M,{-})\:\add(M)\simeq S\modl_\proj
$$
between the full subcategory\/ $\add(M)\subset\sA$ and the full
subcategory of finitely generated projective left $S$\+modules
$S\modl_\proj$ in the category of left $S$\+modules $S\Modl$. \par
\textup{(b)} Let $S=\Hom_\sA(M,M)$ be the endomorphism ring of
the object $M\in\sA$; so the ring $S$ acts on the object $M$
on the left.
 Then the contravariant functor\/ $\Hom_\sA({-},M)\:\sA^\sop\rarrow
S\Modl$ restricts to an anti-equivalence of additive categories
$$
 \Hom_\sA({-},M)\:\add(M)^\sop\simeq S\modl_\proj
$$
between the full subcategory\/ $\add(M)\subset\sA$ and the full
subcategory of finitely generated projective left $S$\+modules
$S\modl_\proj\subset S\Modl$. \qed
\end{lem}

 The following corollary sums up the ``injective coresolution
construction of a bounded below acyclic complex of projective modules''.

\begin{cor} \label{injective-module-coresolution-construction}
 Let $R=k[x_\alpha:\alpha\in A]$ be the ring of polynomials in
infinitely many variables over a field~$k$, and let
\begin{equation} \label{injective-module-coresolution-of-k}
 0\lrarrow k\lrarrow J^0\lrarrow J^1\lrarrow J^2\lrarrow\dotsb
\end{equation}
be an injective coresolution of the one-dimensional $R$\+module~$k$.
 Let $J$ be an injective $R$\+module such that the $R$\+module
$J^n$ is a direct summand of $J$ for all $n\ge0$.
 Let
\begin{equation} \label{acyclic-complex-of-projs-from-inj-cores}
 0\lrarrow\Hom_R(J,J^0)\lrarrow\Hom_R(J,J^1)\lrarrow
 \Hom_R(J,J^2)\lrarrow\dotsb
\end{equation}
be the complex obtained by applying the functor\/ $\Hom_R(J,{-})$ to
the truncated coresolution~\eqref{injective-module-coresolution-of-k}.
 Then~\eqref{acyclic-complex-of-projs-from-inj-cores} is
a bounded below, noncontractible, acyclic complex of finitely
generated projective left modules over the ring $S=\Hom_R(J,J)^\rop$.
\end{cor}

\begin{proof}
 The complex~\eqref{acyclic-complex-of-projs-from-inj-cores} is
acyclic by Theorem~\ref{rickard-theorem}.
 The left $S$\+module $\Hom_R(J,J^n)$ is a direct summand of
the left $S$\+module $\Hom_R(J,J)=S$ for every $n\ge0$, since
the $R$\+module $J^n$ is a direct summand of~$J$.
 So~\eqref{acyclic-complex-of-projs-from-inj-cores} is even
a complex of cyclic projective left $S$\+modules (i.~e.,
projective $S$\+modules with one generator).

 It remains to explain why the complex of
$S$\+modules~\eqref{acyclic-complex-of-projs-from-inj-cores} is not
contractible.
 For this purpose, one observes that, given a full subcategory $\sB$
in an additive category $\sA$, a complex $C^\bu$ in $\sB$ is
contractible in $\sA$ if and only if it is contractible in $\sB$.
 Indeed, any contracting homotopy for $C^\bu$ as a complex in $\sA$
would be a collection of morphisms in $\sA$ between objects from $\sB$,
which means a collection of morphisms in~$\sB$.

 By Lemma~\ref{dress-lemma}(a) (for $\sA=R\Modl$ and $M=J$), the functor
$\Hom_R(J,{-})$ is an equivalence of categories $\add(J)\simeq
S\modl_\proj$.
 The truncated coresolution~\eqref{injective-module-coresolution-of-k},
$$
 0\lrarrow J^0\lrarrow J^1\lrarrow J^2\lrarrow\dotsb
$$
is a noncontractible (since nonacyclic) complex in $R\Modl$ with
the terms belonging to $\add(J)$, so it is a noncontractible
complex in $\add(J)$.
 Applying the equivalence of additive categories $\add(J)\simeq
S\modl_\proj$, we obtain a noncontractible
complex~\eqref{acyclic-complex-of-projs-from-inj-cores}
in $S\modl_\proj$, which is consequently also noncontractible in
$S\Modl$.
 It is important for this argument that the functor
$\Hom_R(J,{-})\:\add(J)\rarrow S\Modl$ is fully faithful.
\end{proof}

\Section{Dual Rickard's Acyclicity Theorem}

 The aim of this section is to prove the following dual version of
Rickard's theorem~\cite[Theorem~3.5]{Rick}.

\begin{thm} \label{dual-rickard}
 Let $R=k[x_\alpha:\alpha\in A]$ be the ring of polynomials in
infinitely many variables over a field~$k$.
 As above, we endow the one-dimensional $k$\+vector space $k$ with
the obvious $R$\+module structure.
 Then, for any flat $R$\+module $P$ and all integers $n\ge0$,
one has\/ $\Ext^n_R(k,P)=0$.
\end{thm}

\begin{proof}
 For every $\alpha\in A$, consider the two-term Koszul complex of
free $R$\+modules with one generator
\begin{equation} \label{two-term-koszul}
 \dotsb\lrarrow 0\lrarrow R\overset{x_\alpha*}\lrarrow R\lrarrow 0
 \lrarrow\dotsb
\end{equation}
situated in the cohomological degrees~$-1$ and~$0$.
 For every finite subset of indices $B\subset A$, denote by
$K^B_\bu(R)$ the tensor product, taken over the ring~$R$, of
the complexes~\eqref{two-term-koszul} with $\alpha\in B$.
 As a finite subset $B\subset A$ varies, the complexes $K^B_\bu(R)$
form an inductive system, indexed by the poset of all finite subsets
$B\subset A$ ordered by inclusion.

 Put $K_\bu(R)=\varinjlim_{B\subset A}K^B_\bu(R)$.
 Then $K_\bu(R)$ is a bounded above complex of free $R$\+modules.
 One has $K_n(R)=0$ for $n<0$, \ $K_0(R)=R$, and $K_n(R)$ is a free
$R$\+module with a set of generators of the cardinality equal to
the cardinality of the set $A$ for all $n>0$.
 (More invariantly, $K_n(R)$ is the free $R$\+module spanned by
the set of all subsets in $A$ of the finite cardinality~$n$).

 For any finite subset $B\subset A$, the complex $K^B_\bu(A)$ is
a finite resolution of the $R$\+module $R/\sum_{\alpha\in B}x_\alpha R$
by finitely generated free $R$\+modules.
 Passing to the direct limit, one can easily see that $K_\bu(R)$ is
a free $R$\+module resolution of the one-dimensional $R$\+module
$k=R/\sum_{\alpha\in A}x_\alpha R$.

 The following three lemmas are straightforward or standard.

\begin{lem} \label{derived-projlim-of-homs}
 Let $R$ be an associative ring, $\Xi$ be a directed poset,
and $(F_\xi)_{\xi\in\Xi}$ be an inductive system of projective
$R$\+modules whose direct limit $F=\varinjlim_{\xi\in\Xi}F_\xi$
is also a projective $R$\+module.
 Let $P$ be an arbitrary $R$\+module.
 Then the higher derived inverse limit functors vanish on
the projective system\/ $\Hom_R(F_\xi,P)_{\xi\in\Xi}$,
$$
 \varprojlim\nolimits_{\xi\in\Xi}\Hom_R(F_\xi,P)=\Hom_R(F,P)
 \quad\text{and}\quad
 \varprojlim\nolimits^i_{\xi\in\Xi}\Hom_R(F_\xi,P)=0
 \quad\text{for all\/ $i\ge1$}.
$$
\end{lem}

\begin{proof}
 Dropping the condition that the $R$\+module
$F=\varinjlim_{\xi\in\Xi}F_\xi$ is projective (but keeping
the conditions that the $R$\+modules $F_\xi$ are projective),
one would have $\varprojlim\nolimits^i_{\xi\in\Xi}\Hom_R(F_\xi,P)
=\Ext^i_R(F,P)$ for every $i\ge0$.
\end{proof}

\begin{lem} \label{derived-projlim-of-eventually-vanishing}
 Let $A$ be an infinite set and $(V_B)_{B\subset A}$ be a projective
system of abelian groups, indexed by the poset of all finite
subsets $B\subset A$ ordered by inclusion.
 Assume that there exists an integer $n\ge0$ such that $V_B=0$
whenever the cardinality of $B$ exceeds~$n$.
 Then the whole derived inverse limit functor vanishes on
the projective system $(V_B)_{B\subset A}$,
$$
 \varprojlim\nolimits^i_{B\subset A} V_B=0
 \qquad\text{for all\/ $i\ge0$}.
$$
\end{lem}

\begin{proof}
 This is a special case of the assertion that the derived functors
of inverse limit are preserved by the passage to a cofinal subsystem.
 This can be deduced from fact that the derived inverse limits
vanish on so-called weakly flabby (faiblement flasque) projective
systems~\cite[Th\'eor\`eme~1.8]{Jen}.
 A stronger result that the derived inverse limit (in an abelian
category with exact product functors) only depends on the pro-object
represented by the given projective system can be found
in~\cite[Corollary~7.3.7]{Pros}.
\end{proof}

\begin{lem} \label{spectral-sequence-lemma}
 Let\/ $\Xi$ be a directed poset and $(C^\bu_\xi)_{\xi\in\Xi}$ be
a projective system of complexes of abelian groups with $C^n_\xi=0$
for $n<0$.
 Then there are two spectral sequences ${}'\!E_r^{pq}$ and
${}''\!E_r^{pq}$, starting from the pages
\begin{alignat*}{2}
 {}'\!E_2^{pq}&=\varprojlim\nolimits^p H^q(C^\bu_\xi), 
 \qquad &&p,\,q\ge0,\\
 {}''\!E_1^{pq}&=\varprojlim\nolimits^q C^p_\xi,
 \qquad && p,\,q\ge0,
\end{alignat*}
with the differentials
${}'\!d_r^{pq}\:{}'\!E_r^{p,q}\rarrow{}'\!E_r^{p+r,q-r+1}$ and
${}''\!d_r^{pq}\:{}''\!E_r^{p,q}\rarrow{}''\!E_r^{p+r,q-r+1}$,
converging to the associated graded groups to two different filtrations
${}'\!F^pE^n$ and ${}''\!F^pE^n$ on the same graded abelian group
$E^n$, \,$n=p+q$.
\end{lem}

\begin{proof}
 These are called ``the two hypercohomology spectral sequences''
(for the derived functor of inverse limit);
cf.~\cite[Section~XVII.3]{CaE}.
 The groups $E^n$ are the cohomology groups of the complex obtained
by applying the derived functor of inverse limit to the whole
complex of projective systems~$(C^\bu_\xi)$.
\end{proof}

 Now we can finish the proof of the theorem.
 By the definition, we have $\Ext^n_R(k,P)=H^n\Hom_R(K_\bu(R),P)$.
 The complex $\Hom_R(K_\bu(R),P)$ is the inverse limit
$$
 \Hom_R(K_\bu(R),P)=\varprojlim\nolimits_{B\subset A}
 \Hom_R(K_\bu^B(R),P).
$$

 For every $n\ge0$, \,Lemma~\ref{derived-projlim-of-homs}
(with the poset $\Xi$ of all finite subsets $B\subset A$,
finitely generated free $R$\+modules $F_B$, and an infinitely
generated free $R$\+module~$F$) tells that
$\varprojlim^i_{B\subset A}\Hom_R(K_n^B(R),P)=0$ for all $i\ge1$.

 On the other hand, the complex $\Hom_R(K_\bu^B(R),P)$ has its only
nonzero cohomology module situated in the cohomological degree~$n$
equal to the cardinality of~$B$ (as $P$ is a flat module over
the ring $R_B=k[x_\alpha:\alpha\in B]$).
 By Lemma~\ref{derived-projlim-of-eventually-vanishing}, we have
$\varprojlim^i_{B\subset A} H^n\Hom_R(K_\bu^B(R),P)=0$
for all $i\ge0$ and $n\ge0$.

 In the context of Lemma~\ref{spectral-sequence-lemma}, put
$C_B^\bu=\Hom_R(K_\bu^B(R),P)$.
 Then we know that ${}'\!E_2^{pq}=\varprojlim^p_{B\subset A}
H^q(C_B^\bu)=0$ for all $p$, $q\ge0$, and
${}''\!E_1^{pq}=\varprojlim^q_{B\subset A}C_B^p=0$ for all $q\ge1$.
 Thus $E^n=0$ and
$H^n(\varprojlim\nolimits_{B\subset A}\Hom_R(K_\bu^B(R),P))
={}''\!E_2^{n,0}=0$ for all $n\ge0$.
\end{proof}

\Section{The Projective Construction of Acyclic Complex of Projectives}

 Now we are ready to present the ``projective resolution construction
of a bounded below acyclic complex of projective modules''.

\begin{cor} \label{projective-module-resolution-construction}
 Let $R=k[x_\alpha:\alpha\in A]$ be the ring of polynomials in
infinitely many variables over a field~$k$, and let
\begin{equation} \label{projective-module-resolution-of-k}
 0\llarrow k\llarrow P_0\llarrow P_1\lrarrow P_2\llarrow\dotsb
\end{equation}
be a projective resolution of the one-dimensional $R$\+module~$k$.
 Let $P$ be a projective $R$\+module such that the $R$\+module
$P_n$ is a direct summand of $P$ for all $n\ge0$.
 Let
\begin{equation} \label{acyclic-complex-of-projs-from-proj-res}
 0\lrarrow\Hom_R(P_0,P)\lrarrow\Hom_R(P_1,P)\lrarrow
 \Hom_R(P_2,P)\lrarrow\dotsb
\end{equation}
be the complex obtained by applying the contravariant functor\/
$\Hom_R({-},P)$ to the truncated
resolution~\eqref{projective-module-resolution-of-k}.
 Then~\eqref{acyclic-complex-of-projs-from-proj-res} is
a bounded below, noncontractible, acyclic complex of finitely
generated projective left modules over the ring $S=\Hom_R(P,P)$.
\end{cor}

\begin{proof}
 The complex~\eqref{acyclic-complex-of-projs-from-proj-res} is acyclic
by Theorem~\ref{dual-rickard}.
 The left $S$\+module $\Hom_R(P_n,P)$ is a direct summand of the left
$S$\+module $\Hom_R(P,P)=S$ for every $n\ge0$, since the $R$\+module
$P_n$ is a direct summand of~$P$.
 So~\eqref{acyclic-complex-of-projs-from-proj-res} is even a complex
of cyclic projective left $S$\+modules.

 The proof of the assertion that the complex of
$S$\+modules~\eqref{acyclic-complex-of-projs-from-proj-res} is not
contractible is similar to the argument in the proof of
Corollary~\ref{injective-module-coresolution-construction}.
 By Lemma~\ref{dress-lemma}(b) (for $\sA=R\Modl$ and $M=P$),
the functor $\Hom_R({-},P)$ is an anti-equivalence of categories
$\add(P)^\sop\simeq S\modl_\proj$.
 The truncated resolution~\eqref{projective-module-resolution-of-k},
$$
 0\llarrow P_0\llarrow P_1\lrarrow P_2\llarrow\dotsb
$$
is a noncontractible (since nonacyclic) complex in $R\Modl$ with
the terms belonging to $\add(P)$, so it is a noncontractible complex
in $\add(P)$.
 Applying the anti-equivalence of additive categories
$\add(P)^\sop\simeq S\modl_\proj$, we obtain a noncontractible
complex~\eqref{acyclic-complex-of-projs-from-proj-res} in
$S\modl_\proj$, which is consequently also noncontractible in
$S\Modl$.
 It is important for this argument that the contravariant functor
$\Hom_R({-},P)\:\add(P)^\sop\lrarrow S\Modl$ is fully faithful.
\end{proof}

\Section{Brief Preliminaries on Coalgebras}

 In this section and the next two, we consider comodules and
contramodules over coassociative, counital coalgebras $\C$
over a field~$k$.
 We refer to the book~\cite{Swe} and the survey
papers~\cite[Section~1]{Prev}, \cite[Sections~3 and~8]{Pksurv} for
background material on coalgebras, comodules, and contramodules.

 For any coalgebra $\C$, there are locally finite Grothendieck
abelian categories of left and right $\C$\+comodules $\C\Comodl$ and
$\Comodr\C$, and a locally presentable abelian category of left
$\C$\+contramodules $\C\Contra$.
 There are enough injective objects in $\C\Comodl$, and they are
precisely the direct summands of the \emph{cofree} left
$\C$\+comodules $\C\ot_k V$ (where $V$ ranges over the $k$\+vector
spaces).
 Dual-analogously, there are enough projective objects in
$\C\Contra$, and they are precisely the direct summands of
the \emph{free} left $\C$\+contramodules $\Hom_k(\C,V)$ (where
$V\in k\Vect$).

 The additive categories of injective left $\C$\+comodules and
projective left $\C$\+con\-tra\-mod\-ules are naturally equivalent,
\begin{equation} \label{underived-co-contra-eqn}
 \Psi_\C\:\C\Comodl_\inj\,\simeq\,\C\Contra_\proj\,:\!\Phi_\C.
\end{equation}
 The equivalence is provided by the restrictions of the adjoint
functors
$$
 \Psi_\C\:\C\comodl\,\leftrightarrows\C\Contra\,:\!\Phi_\C,
$$
the functor $\Phi_\C$ being the left adjoint and $\Psi_\C$
the right adjoint.
 The functors $\Psi_\C$ and $\Phi_\C$ are constructed as
$$
 \Psi_\C(\cM)=\Hom_\C(\C,\cM)
 \quad\text{and}\quad
 \Phi_\C(\fP)=\C\ocn_\C\fP
$$
for all $\cM\in\C\Comodl$ and $\fP\in\C\Contra$.
 Here $\ocn_\C\:\Comodr\C\times\C\Contra\rarrow k\Vect$ is
the functor of \emph{contratensor product} over a coalgebra~$\C$,
while $\Hom_\C$ denotes the Hom functor in the comodule category
$\C\Comodl$.
 The equivalence of additive
categories~\eqref{underived-co-contra-eqn} is called
the (\emph{underived}) \emph{comodule-contramodule correspondence}.
 We refer to~\cite[Sections~1.2 and~3.1]{Prev},
\cite[Sections~8.6\+-8.7]{Pksurv}, or~\cite[Sections~0.2.6
and~5.1]{Psemi} for a more detailed discussion.

 In fact, we are only interested in one special kind of coalgebras,
namely, the \emph{symmetric coalgebra} $\Sym(U)$ of a $k$\+vector
space~$U$.
 To define the symmetric coalgebra, consider the \emph{tensor coalgebra}
$\Ten(U)=\bigoplus_{n=0}^\infty U^{\ot n}$, as defined, e.~g.,
in~\cite[Section~2.3]{Pksurv} (where the notation is slightly
different).
 The tensor coalgebra is the cofree conilpotent coalgebra cospanned
by~$U$ \,\cite[Remark~3.2]{Pksurv}; it is also naturally graded.
 The symmetric coalgebra is simplest defined as the graded
subcoalgebra in $\Ten(U)$ whose grading components $\Sym_n(U)
\subset\Ten_n(U)=U^{\ot n}$ are the subspaces of symmetric tensors
$\Sym_n(U)\subset U^{\ot n}$ in the tensor powers of the vector
space~$U$.
 So the whole symmetric coalgebra is $\Sym(U)=\bigoplus_{n=0}^\infty
\Sym_n(U)=k\oplus U\oplus\Sym_2(U)\oplus\dotsb$.

 Following the discussion in~\cite[Section~1.3\+-1.4]{Prev}
or~\cite[Section~8.3]{Pksurv}, coalgebras $\C$ can be described (and in
fact, defined) in terms of their vector space dual algebras $\C^*=
\Hom_k(\C,k)$, which carry natural linearly compact
(\,$=$~pseudocompact) topologies.
 In particular, if $U$ is a finite-dimensional $k$\+vector space
with a basis $x_1^*$,~\dots,~$x_m^*$, then the dual algebra
$\Sym(U)^*$ to the symmetric coalgebra $\Sym(U)$ is the topological
algebra of formal Taylor power series $\Sym(U)^*=k[[x_1,\dotsc,x_m]]$.

 Generally speaking, for an infinite-dimensional $k$\+vector space
$W$, one has $\Sym(W)=\varinjlim_{U\subset W}\Sym(U)$ and
$\Sym(W)^*=\varprojlim_{U\subset W}\Sym(U)^*$, where $U$ ranges over
the finite-dimensional vector subspaces of~$W$.
 So, if $\{x_\alpha^*:\alpha\in A\}$ is a $k$\+vector space basis
of $W$, indexed by some set $A$, then
$\Sym(W)^*=\varprojlim_{B\subset A}k[[x_\alpha:\alpha\in B]]$,
where $B$ ranges over the finite subsets of~$A$.
 Here, given two finite subsets $B'\subset B''\subset A$, the transition
map $k[[x_\alpha:\alpha\in B'']]\rarrow k[[x_\alpha:\alpha\in B']]$
in the projective system takes $x_\alpha$ to~$x_\alpha$ for
all $\alpha\in B'$ and $x_\beta$ to~$0$ for all $\beta\in
B''\setminus B'$.
 Such rings
$\Sym(W)^*=\varprojlim_{B\subset A}k[[x_\alpha:\alpha\in B]]$ are
the ``commutative rings of formal power series in infinitely many
variables'' that we are interested in.
{\hbadness=1100\par}

\Section{Comodule and Contramodule Acyclicity Theorems}

 As above, we denote by $W$ an infinite-dimensional $k$\+vector space
with a basis $\{x_\alpha^*:\alpha\in A\}$ indexed by a set~$A$.
 Given a finite set $B$, we let $\R_B=k[[x_\alpha:\alpha\in B]]$ be
the (topological) ring of commutative formal Taylor power series in
finitely many variables indexed by~$B$.
 Furthermore, we put $\R=\varprojlim_{B\subset A}\R_B$ (with
the transition maps described in the previous section).
 So, denoting by $U_B\subset W$ the finite-dimensional vector subspace
spanned by $\{x_\alpha^*:\alpha\in B\}$, we have
$\R_B=\Sym(U_B)^*$ and
$\R=\varprojlim_{B\subset A}k[[x_\alpha:\alpha\in B]]=\Sym(W)^*$.
 Let us also introduce the notation $\C_B=\Sym(U_B)$ and
$\C=\Sym(W)$ for the symmetric coalgebras.

 As in the proof of Theorem~\ref{dual-rickard}, we start with
considering the two-term Koszul complex of free $\R_B$\+modules
with one generator
\begin{equation} \label{two-term-koszul-over-R_B}
 \dotsb\lrarrow 0\lrarrow\R_B\overset{x_\alpha*}\lrarrow\R_B\lrarrow 0
 \lrarrow\dotsb
\end{equation}
situated in the cohomological degrees~$-1$ and~$0$
(where $\alpha\in B$).
 Denote by $K^B_\bu(\R_B)$ the tensor product, taken over the ring
$\widehat R_B$, of the complexes~\eqref{two-term-koszul-over-R_B}.
 As the elements $\{x_\alpha:\alpha\in B\}$ form a regular sequence
in the formal power series ring $\R_B$, the complex $K^B_\bu(\R_B)$
is a finite resolution of the one-dimensional $\R_B$\+module
$k=\R_B/\sum_{\alpha\in B}x_\alpha\R_B$ by finitely generated
free $\R_B$\+modules.

 The (augmented) Koszul complex $K^B_\bu(\R_B)\rarrow k$ is a complex
of linearly compact topological $k$\+vector spaces; so it can be
obtained by applying the vector space dualization functor
$\Hom_k({-},k)$ to a certain complex of discrete vector spaces.
 The latter complex has the form
\begin{multline} \label{fin-cogen-symm-comodule-coresolution}
 0\lrarrow k\lrarrow\Sym(U_B)\lrarrow\Sym(U_B)\ot_k U_B \\
 \lrarrow \Sym(U_B)\ot_k\Lambda^2(U_B) \lrarrow\dotsb\lrarrow
 \Sym(U_B)\ot_k\Lambda^m(U_B)\lrarrow0,
\end{multline}
where $m=\dim U_B$ and $\Lambda^n(V)$, \,$n\ge0$, denotes the exterior
powers of a vector space~$V$.
 The complex~\eqref{fin-cogen-symm-comodule-coresolution} is
an injective/cofree $\C_B$\+comodule coresolution of the trivial
one-dimensional comodule~$k$ over the conilpotent coalgebra
$\C_B=\Sym(U_B)$.

 Passing to the direct limit of
the finite complexes~\eqref{fin-cogen-symm-comodule-coresolution} over
all the finite subsets $B\subset A$, we obtain a bounded below complex
\begin{multline} \label{infinite-symm-comodule-coresolution}
 0\lrarrow k\lrarrow\Sym(W)\lrarrow\Sym(W)\ot_k W \\
 \lrarrow \Sym(W)\ot_k\Lambda^2(W) \lrarrow\dotsb\lrarrow
 \Sym(W)\ot_k\Lambda^n(W)\lrarrow\dotsb
\end{multline}
 The complex~\eqref{infinite-symm-comodule-coresolution} is
an injective/cofree $\C$\+comodule coresolution of the trivial
one-dimensional comodule~$k$ over the conilpotent coalgebra
$\C=\Sym(W)$.

 One can easily check that
the coresolutions~\eqref{fin-cogen-symm-comodule-coresolution}
and~\eqref{infinite-symm-comodule-coresolution} are well-defined
and functorial for any $k$\+vector spaces $U$ (in place of~$U_B$)
and $W$, and do not depend on the choice of any bases in the vector
spaces.
 In fact, the differential $\Sym(W)\ot_k\Lambda^n(W)\rarrow
\Sym(W)\ot_k\Lambda^{n+1}(W)$ can be constructed as the composition
$\Sym(W)\ot_k\Lambda^n(W)\rarrow\Sym(W)\ot_kW\ot_k\Lambda^n(W)
\rarrow\Sym(W)\ot_k\Lambda^{n+1}(W)$ of the map induced
by the comultiplication map $\Sym(W)\rarrow\Sym(W)\ot_kW$ and
the map induced by the multiplication map $W\ot_k\Lambda^n(W)
\rarrow\Lambda^{n+1}(W)$.

 Applying the vector space dualization functor $\Hom_k({-},k)$ to
the complex~\eqref{infinite-symm-comodule-coresolution}, we obtain
a bounded above complex
\begin{multline} \label{infinite-symm-contramodule-resolution}
 0\llarrow k\llarrow\Hom_k(\C,k)\llarrow\Hom_k(\C,W^*) \\
 \llarrow \Hom_k(\C,\Lambda^2(W)^*) \llarrow\dotsb\llarrow
 \Hom_k(\C,\Lambda^n(W)^*)\llarrow\dotsb
\end{multline}
 The complex~\eqref{infinite-symm-contramodule-resolution} is
a projective/free $\C$\+contramodule resolution of the trivial
one-dimensional $\C$\+contramodule~$k$.

 Applying the functor $\Phi_\C=\C\ocn_\C{-}$ to the truncated
$\C$\+contramodule
resolution~\eqref{infinite-symm-contramodule-resolution}, we obtain
a bounded above complex of injective/cofree $\C$\+comodules
\begin{multline} \label{acyclic-comodule-complex}
 0\llarrow\C\llarrow\C\ot_kW^*\llarrow\C\ot_k\Lambda^2(W)^* \\
 \llarrow\dotsb\llarrow\C\ot_k\Lambda^n(W)^*\llarrow\dotsb
\end{multline}
 Applying the functor $\Psi_\C=\Hom_\C(\C,{-})$ to the truncated
$\C$\+comodule coresolution~\eqref{infinite-symm-comodule-coresolution},
we obtain a bounded below complex of projective/free $\C$\+contramodules
\begin{multline} \label{acyclic-contramodule-complex}
 0\lrarrow\Hom_k(\C,k)\lrarrow\Hom_k(\C,W) \\
 \lrarrow \Hom_k(\C,\Lambda^2(W)) \lrarrow\dotsb\lrarrow
 \Hom_k(\C,\Lambda^n(W))\lrarrow\dotsb
\end{multline}

\begin{thm} \label{comodule-complex-is-acyclic-theorem}
 For any infinite-dimensional $k$\+vector space $W$,
the complex of cofree comodules~\eqref{acyclic-comodule-complex} is
acyclic (i.~e., its cohomology spaces vanish in all the degrees).
\end{thm}

\begin{proof}
 This was stated in~\cite[Section~0.2.7]{Psemi} (as a part of
introductory/preliminary material for the book).
 The proof is not difficult.

 The complex~\eqref{acyclic-comodule-complex} is the direct limit of
its subcomplexes
\begin{multline} \label{complex-of-comodules-over-fin-dim}
 0\llarrow\C_B\llarrow\C_B\ot_kW^*\llarrow\C_B\ot_k\Lambda^2(W)^* \\
 \llarrow\dotsb\llarrow\C_B\ot_k\Lambda^n(W)^*\llarrow\dotsb
\end{multline}
taken over the directed poset of all finite subsets $B\subset A$.
 The complex~\eqref{complex-of-comodules-over-fin-dim}, which is
a complex of comodules over the subcoalgebra $\C_B=\Sym(U_B)$
of the coalgebra $\C=\Sym(W)$, can be obtained by applying
the cotensor product functor $\C_B\oc_\C{-}$ to
the complex~\eqref{acyclic-comodule-complex}
(see~\cite[Sections~2.5\+-2.6]{Prev} or~\cite[Section~0.2.1
or~1.2.1]{Psemi}).

 The complex~\eqref{complex-of-comodules-over-fin-dim} is \emph{not}
acyclic, but its cohomology spaces gradually vanish as the size of
the finite subset $B\subset A$ grows.
 Indeed, applying the vector space dualization functor $\Hom_k({-},k)$
to the finite complex~\eqref{fin-cogen-symm-comodule-coresolution},
we obtain a finite Koszul complex that was denoted above
by~$K^B_\bu(\R_B)$.
 It has the form
\begin{multline} \label{fin-cogen-symm-contramodule-resolution}
 0\llarrow k\llarrow\Hom_k(\C_B,k)\llarrow\Hom_k(\C,U_B^*) \\
 \llarrow \Hom_k(\C_B,\Lambda^2(U_B)^*) \llarrow\dotsb\llarrow
 \Hom_k(\C_B,\Lambda^m(U_B)^*)\llarrow 0
\end{multline}
and can be viewed as a projective/free $\C_B$\+contramodule resolution
of the trivial one-dimensional $\C_B$\+contramodule~$k$.

 Applying the functor $\Phi_{\C_B}=\C_B\ocn_{\C_B}{-}$ to the truncated
$\C_B$\+contramodule
resolution~\eqref{fin-cogen-symm-contramodule-resolution}, we obtain
a finite complex of injective/cofree $\C_B$\+comodules
\begin{multline} \label{dual-Koszul-of-comodules-over-fin-dim}
 0\llarrow\C_B\llarrow\C_B\ot_kU_B^*\llarrow\C_B\ot_k\Lambda^2(U_B)^*
 \\ \llarrow\dotsb\llarrow\C_B\ot_k\Lambda^m(U_B)^*\llarrow 0.
\end{multline}
 The only cohomology space of
the complex~\eqref{dual-Koszul-of-comodules-over-fin-dim} is
the one-dimensional $k$\+vector space $\Lambda^m(U_B)^*$ situated
in the cohomological degree~$-m$, i.~e., at the rightmost term.

 Consider the direct sum decomposition $W=U_B\oplus V_B$,
where $V_B\subset W$ is the subspace with the basis
$\{x_\alpha^*:\alpha\in A\setminus B\}$.
 Consider the graded dual vector space to the exterior algebra
$\bigoplus_{n=0}^\infty\Lambda^n(V_B)$, and view it as a complex
\begin{equation} \label{zero-differential-dual-exterior-powers}
 0\llarrow k\overset 0\llarrow V_B^*\overset 0\llarrow
 \Lambda^2(V_B)^*\overset0\llarrow\dotsb\overset0\llarrow
 \Lambda^n(V_B)^*\overset0\llarrow\dotsb
\end{equation}
with zero differential.
 Then the complex~\eqref{complex-of-comodules-over-fin-dim} is
the tensor product, taken over the field~$k$, of
the complexes~\eqref{dual-Koszul-of-comodules-over-fin-dim}
and~\eqref{zero-differential-dual-exterior-powers}.
 Accordingly, the cohomology spaces of
the complex~\eqref{complex-of-comodules-over-fin-dim} are concentrated
in the cohomological degrees~$\le -m$, where $m$~is the cardinality
of the set~$B$.

 As the size of the subset $B\subset A$ grows, the cohomology of
the complex~\eqref{complex-of-comodules-over-fin-dim} move away
and disappear at the cohomological degree~$-\infty$.
 So the direct limit~\eqref{acyclic-comodule-complex} of
the complexes~\eqref{complex-of-comodules-over-fin-dim} is acyclic.
\end{proof}

\begin{thm} \label{contramodule-complex-is-acyclic-theorem}
 For any infinite-dimensional $k$\+vector space $W$,
the complex of free contramodules~\eqref{acyclic-contramodule-complex}
is acyclic (i.~e., its cohomology spaces vanish in all the degrees).
\end{thm}

\begin{proof}
 This was also stated in~\cite[Section~0.2.7]{Psemi}.
 The proof is only slightly more complicated than the proof of
the previous theorem, in that one needs to deal with inverse limits.
 However, we have done all the preparatory work already.

 The complex~\eqref{acyclic-contramodule-complex} is the inverse limit
of its quotient complexes
\begin{multline} \label{complex-of-contramodules-over-fin-dim}
 0\lrarrow\Hom_k(\C_B,k)\lrarrow\Hom_k(\C_B,W) \\
 \lrarrow \Hom_k(\C_B,\Lambda^2(W)) \lrarrow\dotsb\lrarrow
 \Hom_k(\C_B,\Lambda^n(W))\lrarrow\dotsb
\end{multline}
taken over the directed poset of all finite subsets $B\subset A$.
 The complex~\eqref{complex-of-contramodules-over-fin-dim}, which
is a complex of contramodules over the subcoalgebra $\C_B\subset\C$,
can be obtained by applying the Cohom functor $\Cohom_\C(\C_B,{-})$
to the complex~\eqref{acyclic-contramodule-complex}
(see~\cite[Sections~2.5\+-2.6]{Prev}
or~\cite[Section~0.2.4 or~3.2.1]{Psemi}).

 Similarly to the previous proof,
the complex~\eqref{complex-of-contramodules-over-fin-dim} is \emph{not}
acyclic, but its cohomology spaces gradually vanish as the size of
the finite subset $B\subset A$ grows.
 For the sake of completeness of the exposition, let us start with 
applying the functor $\Psi_{\C_B}=\Hom_{\C_B}(\C_B,{-})$ to
the truncated $\C_B$\+comodule
coresolution~\eqref{fin-cogen-symm-comodule-coresolution}.
 We obtain
a finite complex of projective/free $\C_B$\+contramodules
\begin{multline} \label{dual-Koszul-of-contramodules-over-fin-dim}
 0\lrarrow\Hom_k(\C_B,k)\lrarrow\Hom_k(\C_B,U_B) \\
 \lrarrow\Hom_k(\C_B,\Lambda^2(U_B))\lrarrow\dotsb\lrarrow
 \Hom_k(\C_B,\Lambda^m(U_B))\lrarrow0.
\end{multline} 
 The only cohomology space of
the complex~\eqref{dual-Koszul-of-contramodules-over-fin-dim} is
the one-dimensional $k$\+vector space $\Lambda^m(U_B)$ situated
in the cohomological degree~$m$, i.~e, at the rightmost term.
 In fact, the complex of
contramodules~\eqref{dual-Koszul-of-contramodules-over-fin-dim}
can be obtained by applying the vector space dualization functor
$\Hom_k({-},k)$ to the complex of
comodules~\eqref{dual-Koszul-of-comodules-over-fin-dim}.

 Consider the exterior algebra $\bigoplus_{n=0}^\infty
\Lambda^n(V_B)$, where as in the previous proof $W=U_B\oplus V_B$,
and view it as a complex
\begin{equation} \label{zero-differential-exterior-powers}
 0\lrarrow k\overset 0\lrarrow V_B\overset 0\lrarrow
 \Lambda^2(V_B)\overset0\lrarrow\dotsb\overset0\lrarrow
 \Lambda^n(V_B)\overset0\lrarrow\dotsb
\end{equation}
with zero differential.
 Then the complex~\eqref{complex-of-contramodules-over-fin-dim} is
the complex of $k$\+vector space morphisms, $\Hom_k({-},{-})$, from
the complex~\eqref{dual-Koszul-of-comodules-over-fin-dim} into
the complex~\eqref{zero-differential-exterior-powers}.
 Accordingly, the cohomology spaces of
the complex~\eqref{complex-of-contramodules-over-fin-dim} are
concentrated in the cohomological degrees~$\ge m$.

 The rest of the argument proceeds along the lines of the proof of
Theorem~\ref{dual-rickard}, based on
Lemmas~\ref{derived-projlim-of-homs}\+-\ref{spectral-sequence-lemma}.
 As mentioned above, the complex~\eqref{acyclic-contramodule-complex}
is the inverse limit of
the complexes~\eqref{complex-of-contramodules-over-fin-dim}
taken over the directed poset $\Xi$ of all finite subsets $B\subset A$
with respect to inclusion.
 At every cohomological degree~$n\ge0$,
\,Lemma~\ref{derived-projlim-of-homs} (for $R=k$, \
$F_B=\C_B$, and $P=\Lambda^n(W)$) tells that
$\varprojlim^i_{B\subset A}\Hom_k(\C_B,\Lambda^n(W))=0$
for all $i\ge1$.

 Denote by $C_B^\bu$
the complex~\eqref{complex-of-contramodules-over-fin-dim}.
 By Lemma~\ref{derived-projlim-of-eventually-vanishing}, we have
$\varprojlim^i_{B\subset A}H^n(C_B^\bu)=0$ for all $i\ge0$ and $n\ge0$.
 Now in the context of Lemma~\ref{spectral-sequence-lemma} we have
${}'\!E_2^{pq}=0$ for all $p$, $q\ge0$, and ${}''\!E_1^{pq}=0$ for
all $q\ge1$.
 Therefore, $E^n=0$ and $H^n(\varprojlim_{B\subset A}C_B^\bu)=
{}''\!E_2^{n,0}=0$ for all $n\ge0$.
\end{proof}

\Section{Two Contramodule Constructions of Acyclic Complexes of
Projectives}

 We have essentially already constructed the promised bounded above,
noncontractible, acyclic complex of injective comodules and
bounded below, noncontractible, acyclic complex of projective
contramodules over the cocommutative coalgebra $\C=\Sym(W)$.
 Let us state this as a corollary.

\begin{cor} \label{acyclic-of-inj-comod-proj-contra-cor}
 Let $W$ be an infinite-dimensional vector space over a field~$k$ and\/
$\C=\Sym(W)$ be the symmetric coalgebra.
 Then \par
\textup{(a)} the complex~\eqref{acyclic-comodule-complex} is a bounded
above, noncontractible, acyclic complex of injective comodules
over\/~$\C$; \par
\textup{(b)} the complex~\eqref{acyclic-contramodule-complex} is
a bounded below, noncontractible, acyclic complex of projective
contramodules over\/~$\C$.
\end{cor}

\begin{proof}
 Part~(a): the complex~\eqref{acyclic-comodule-complex} is acyclic
by Theorem~\ref{comodule-complex-is-acyclic-theorem}.
 It remains to explain why the complex of
$\C$\+comodules~\eqref{acyclic-comodule-complex} is not contractible.

 The truncated
resolution~\eqref{infinite-symm-contramodule-resolution},
\begin{multline*} 
 0\llarrow\Hom_k(\C,k)\llarrow\Hom_k(\C,W^*) \\
 \llarrow \Hom_k(\C,\Lambda^2(W)^*) \llarrow\dotsb\llarrow
 \Hom_k(\C,\Lambda^n(W)^*)\llarrow\dotsb
\end{multline*}
is a noncontractible (since nonacyclic) complex in the abelian
category $\C\Contra$ with the terms belonging to the full subcategory
of projective objects $\C\Contra_\proj$, so it is a noncontractible
complex in $\C\Contra_\proj$.
 Applying the equivalence of additive categories
$\Phi_\C\:\C\Contra_\proj\simeq\C\Comodl_\inj$
\,\eqref{underived-co-contra-eqn}, we obtain a noncontractible
complex~\eqref{acyclic-comodule-complex} in $\C\Comodl_\inj$,
which is consequently also noncontractible in $\C\Comodl$.

 Part~(b): the complex~\eqref{acyclic-contramodule-complex} is acyclic
by Theorem~\ref{contramodule-complex-is-acyclic-theorem}.
 It remains to explain why the complex of
$\C$\+contramodules~\eqref{acyclic-contramodule-complex}
is not contractible.

 The truncated
coresolution~\eqref{infinite-symm-comodule-coresolution},
$$
 0\lrarrow\C\lrarrow\C\ot_k W
 \lrarrow \C\ot_k\Lambda^2(W) \lrarrow\dotsb\lrarrow
 \C\ot_k\Lambda^n(W)\lrarrow\dotsb
$$
is a noncontractible (since nonacyclic) complex in the abelian
category $\C\Comodl$ with the terms belonging to the full subcategory
of injective objects $\C\Comodl_\inj$, so it is a noncontractible
complex in $\C\Comodl_\inj$.
 Applying the equivalence of additive categories
$\Psi_\C\:\C\Comodl_\inj\simeq\C\Contra_\proj$
\,\eqref{underived-co-contra-eqn}, we obtain a noncontractible
complex~\eqref{acyclic-contramodule-complex} in $\C\Contra_\proj$,
which is consequently also noncontractible in $\C\Contra$.
\end{proof}

 Now let us present the two contramodule constructions of bounded
below, noncontractible, acyclic complexes of projective modules.
 Recall the notation $\Hom_\C({-},{-})$ for the Hom spaces in
the category $\C\Comodl$.
 The notation $\Hom^\C({-},{-})$ stands for the Hom spaces in
the category $\C\Contra$.

\begin{cor} \label{first-contramodule-construction-cor}
 Let $W$ be an infinite-dimensional vector space over a field~$k$ and\/
$\C=\Sym(W)$ be the symmetric coalgebra.
 Let
$$
 0\lrarrow k\lrarrow\J^0\lrarrow\J^1\lrarrow\J^2\lrarrow\dotsb
$$
be a notation for the injective
coresolution~\eqref{infinite-symm-comodule-coresolution}
of the trivial one-dimensional\/ $\C$\+comodule~$k$.
 Denote by\/ $\J$ the cofree\/ $\C$\+comodule\/ $\C\ot_kW$ cospanned
by~$W$.
 Let
\begin{equation} \label{acyclic-complex-of-projs-from-comodule-cores}
 0\lrarrow\Hom_\C(\J,\J^0)\lrarrow\Hom_\C(\J,\J^1)\lrarrow
 \Hom_\C(\J,\J^2)\lrarrow\dotsb
\end{equation}
be the complex obtained by applying the functor\/ $\Hom_\C(\J,{-})$ to
the truncated coresolution~\eqref{infinite-symm-comodule-coresolution}.
 Then~\eqref{acyclic-complex-of-projs-from-comodule-cores} is
a bounded below, noncontractible, acyclic complex of finitely generated
projective left modules over the ring $S=\Hom_\C(\J,\J)^\rop$.
\end{cor}

\begin{proof}
 For any $\C$\+comodule $\cM$, we have $\Hom_\C(\C\ot_kW,\>\cM)=
\Hom_k(W,\Hom_k(\C,\cM))\allowbreak=\Hom_k(W,\Psi_\C(\cM))$.
 Thus the complex~\eqref{acyclic-complex-of-projs-from-comodule-cores}
can be obtained by applying the vector space Hom functor
$\Hom_k(W,{-})$ to the complex~\eqref{acyclic-contramodule-complex},
and it follows from
Theorem~\ref{contramodule-complex-is-acyclic-theorem} that
the complex~\eqref{acyclic-complex-of-projs-from-comodule-cores}
is acyclic.

 Furthermore, by construction, the $\C$\+comodule $\J^0=\C$ is
a direct summand of $\J$, while the $\C$\+comodules
$\J^n=\C\ot_k\Lambda^n(W)$ are isomorphic to $\J$ for $n\ge1$.
 Hence the left $S$\+module $\Hom_\C(\J,\J^0)$ is a direct summand of
$\Hom_\C(\J,\J)=S$, and the left $S$\+modules $\Hom_\C(\J,\J^n)$
are isomorphic to $S$ for $n\ge1$.
 So~\eqref{acyclic-complex-of-projs-from-comodule-cores} is even
a complex of cyclic projective left $S$\+modules.

 The assertion that the complex of
$S$\+modules~\eqref{acyclic-complex-of-projs-from-comodule-cores} is
not contractible is provable similarly to the argument in
the proof of Corollary~\ref{injective-module-coresolution-construction}.
 By Lemma~\ref{dress-lemma}(a) (for $\sA=\C\Comodl$ and $M=\J$),
the functor $\Hom_\C(\J,{-})$ is an equivalence of categories
$\add(\J)\simeq S\modl_\proj$.
 The truncated coresolution~\eqref{infinite-symm-comodule-coresolution},
$$
 0\lrarrow\J^0\lrarrow\J^1\lrarrow\J^2\lrarrow\dotsb
$$
is a noncontractible (since nonacyclic) complex in $\C\Comodl$ with
the terms belonging to $\add(\J)$, so it is a noncontractible complex
in $\add(\J)$.
 Applying the equivalence of additive categories $\add(\J)\simeq
S\modl_\proj$, we obtain a noncontractible
complex~\eqref{acyclic-complex-of-projs-from-comodule-cores} in
$S\modl_\proj$, which is consequently also noncontractible in $S\Modl$.
 It is important for this argument that the functor
$\Hom_\C(\J,{-})\:\add(\J)\rarrow S\Modl$ is fully faithful.

 Alternatively, put $\fP=\Hom_k(\C,W)\in\C\Contra_\proj$.
 Then the co-contra correspondence~\eqref{underived-co-contra-eqn}
restricts to an equivalence of additive categories
$\add(\J)\simeq\add(\fP)$ taking $\J$ to~$\fP$.
 Hence the ring $S$ can be alternatively described as
$S=\Hom^\C(\fP,\fP)^\rop$.
 The complex of left
$S$\+modules~\eqref{acyclic-complex-of-projs-from-comodule-cores}
can be constructed by applying the functor $\Hom^\C(\fP,{-})$ to
the complex of $\C$\+contramodules~\eqref{acyclic-contramodule-complex},
whose terms belong to $\add(\fP)$.

 Then the noncontractibility argument can be based on
Corollary~\ref{acyclic-of-inj-comod-proj-contra-cor}(b) and
the fact that the functor $\Hom^\C(\fP,{-})$ is an equivalence
of categories $\add(\fP)\simeq S\modl_\proj$
(by Lemma~\ref{dress-lemma}(a) for $\sA=\C\Contra$ and $M=\fP$).
 Once again, it is important for this argument that the functor
$\Hom^\C(\fP,{-})\:\add(\fP)\rarrow S\Modl$ is fully faithful.
 In fact, the whole functor $\Hom^\C(\fP,{-})\:\C\Contra\rarrow
S\Modl$ is fully faithful (on the whole abelian category $\C\Contra$)
by~\cite[Theorem~6.10]{PS1}.
 The latter conclusion is based on the observations that $\fP$ is
the coproduct of $\dim W$ copies of the projective generator
$\C^*=\Hom_k(\C,k)$ of the abelian category $\C\Contra$, and
$\C^*$ is abstractly $\kappa$\+small in $\C\Contra$ for $\C=\Sym(W)$
if $\kappa$~is the successor cardinality of $\dim W$.
\end{proof}

\begin{cor} \label{second-contramodule-construction-cor}
 Let $W$ be an infinite-dimensional vector space over a field~$k$ and\/
$\C=\Sym(W)$ be the symmetric coalgebra.
 Let
$$
 0\llarrow k\llarrow\fP_0\llarrow\fP_1\llarrow\fP_2\llarrow\dotsb
$$
be a notation for the projective
resolution~\eqref{infinite-symm-contramodule-resolution}
of the trivial one-dimensional\/ $\C$\+con\-tra\-mod\-ule~$k$.
 Denote by\/ $\fP$ the free\/ $\C$\+contramodule\/ $\Hom_k(\C,W^*)$
spanned by the vector space $W^*=\Hom_k(W,k)$.
 Let
\begin{equation} \label{acyclic-complex-of-projs-from-contramodule-res}
 0\lrarrow\Hom^\C(\fP_0,\fP)\lrarrow\Hom^\C(\fP_1,\fP)\lrarrow
 \Hom^\C(\fP_2,\fP)\lrarrow\dotsb
\end{equation}
be the complex obtained by applying the contravariant
functor\/ $\Hom^\C({-},\fP)$ to the truncated
resolution~\eqref{infinite-symm-contramodule-resolution}.
 Then~\eqref{acyclic-complex-of-projs-from-contramodule-res} is
a bounded below, noncontractible, acyclic complex of finitely generated
projective left modules over the ring $S=\Hom^\C(\fP,\fP)$.
\end{cor}

\begin{proof}
 For any right $\C$\+comodule $\cN$, any left $\C$\+contramodule $\fQ$,
and any $k$\+vector space $V$, there is a natural isomorphism of
$k$\+vector spaces
$$
 \Hom^\C(\fQ,\Hom_k(\cN,V))\simeq\Hom_k(\cN\ocn_\C\fQ,\>V)
$$
\cite[Section~3.1]{Prev}, \cite[Section~8.6]{Pksurv},
or~\cite[Sections~0.2.6 and~5.1.1]{Psemi}.
 In particular, we have natural isomorphisms
$$
 \Hom^\C(\fQ,\fP)=\Hom_k(\C\ocn_\C\fQ,\>W^*)=
 \Hom_k(\Phi_\C(\fQ),W^*).
$$
 Thus the complex~\eqref{acyclic-complex-of-projs-from-contramodule-res}
can be obtained by applying the contravariant vector space Hom functor
$\Hom_k({-},W^*)$ to the complex~\eqref{acyclic-comodule-complex}, and
it follows from Theorem~\ref{comodule-complex-is-acyclic-theorem} that
the complex~\eqref{acyclic-complex-of-projs-from-contramodule-res}
is acyclic.

 Furthermore, by construction, the $\C$\+comodule $\fP_0=\C^*$ is
is a direct summand of $\fP$, while the $\C$\+contramodules
$\fP_n=\Hom_k(\C,\Lambda^n(W)^*)$ are isomorphic to $\fP$ for $n\ge1$.
 Hence the left $S$\+module $\Hom^\C(\fP_0,\fP)$ is a direct summand
of $\Hom^\C(\fP,\fP)=S$, and the left $S$\+modules $\Hom^\C(\fP_n,\fP)$
are isomorphic to $S$ for $n\ge1$.
 So~\eqref{acyclic-complex-of-projs-from-contramodule-res} is even
a complex of cyclic projective left $S$\+modules.

 The assertion that the complex of
$S$\+modules~\eqref{acyclic-complex-of-projs-from-contramodule-res} is
not contractible is provable similarly to the argument in the proof
of Corollary~\ref{projective-module-resolution-construction}.
 By Lemma~\ref{dress-lemma}(b) (for $\sA=\C\Contra$ and $M=\fP$),
the functor $\Hom^\C({-},\fP)$ is an anti-equivalence of categories
$\add(\fP)^\sop\simeq S\modl_\proj$.
 The truncated resolution~\eqref{infinite-symm-contramodule-resolution},
$$
 0\llarrow\fP_0\llarrow\fP_1\llarrow\fP_2\llarrow\dotsb
$$
is a noncontractible (since nonacyclic) complex in $\C\Contra$ with
the terms belonging to $\add(\fP)$, so it is a noncontractible complex
in $\add(\fP$).
 Applying the anti-equivalence of additive categories
$\add(\fP)^\sop\simeq S\modl_\proj$, we obtain a noncontractible
complex~\eqref{acyclic-complex-of-projs-from-contramodule-res}, which
is consequently also noncontractible in $S\Modl$.
 It is important for this argument that the contravariant functor
$\Hom^\C({-},\fP)\:\add(\fP)^\sop\rarrow S\Modl$ is fully faithful.
\end{proof}

\Section{Summary of the Examples Obtained}

 Now we can summarize our constructions as follows.

\begin{conc} \label{main-conclusion}
 There exists an associative ring $S$ for which \par
\textup{(a)} there is a bounded above acyclic complex of injective
right $S$\+modules that is not contractible; \par
\textup{(b)} there is a bounded below acyclic complex of flat left
$S$\+modules that is not pure acyclic; \par
\textup{(c)} there is a bounded below acyclic complex of (finitely
generated) projective left $S$\+modules that is not contractible.
\end{conc}

\begin{proof}
 Proposition~\ref{projective-is-enough-to-refute} tells that any
ring $S$ satisfying~(c) also satisfies~(a) and~(b).
 Various examples of associative rings $S$ satisfying~(c) are provided
by Corollaries~\ref{injective-module-coresolution-construction},
\ref{projective-module-resolution-construction},
\ref{first-contramodule-construction-cor},
and~\ref{second-contramodule-construction-cor}.
\end{proof}

 What can one say about the rings $S$ appearing in
Corollaries~\ref{injective-module-coresolution-construction},
\ref{projective-module-resolution-construction},
\ref{first-contramodule-construction-cor},
and~\ref{second-contramodule-construction-cor}\,?
 First of all, \emph{none} of them is commutative
(while we have cocommutative coalgebra examples in
Corollary~\ref{acyclic-of-inj-comod-proj-contra-cor}).

 Let us denote the respective versions of the ring $S$ by
$S_{\ref{injective-module-coresolution-construction}}$,
$S_{\ref{projective-module-resolution-construction}}$,
$S_{\ref{first-contramodule-construction-cor}}$,
and $S_{\ref{second-contramodule-construction-cor}}$.
 While the ring $S_{\ref{injective-module-coresolution-construction}}$
(from Corollary~\ref{injective-module-coresolution-construction})
appears to be complicated and hard to visualize, the rings
$S_{\ref{projective-module-resolution-construction}}$,
$S_{\ref{first-contramodule-construction-cor}}$,
and $S_{\ref{second-contramodule-construction-cor}}$ can be described
rather explicitly.

 In the context of
Corollary~\ref{projective-module-resolution-construction}, it makes
sense to choose the infinite Koszul complex
$K_\bu(R)=\varinjlim_{B\subset A}K^B_\bu(R)$ to play the role of
the projective resolution $P_\bu$
\,\eqref{projective-module-resolution-of-k} of the $R$\+module~$k$.
 In this case, one can take $P$ to be the free $R$\+module with $A$
generators, $P=\bigoplus_{\alpha\in A}R$.
 Then the $R$\+module $P_0=R$ is a direct summand of $P$, while
the $R$\+module $P_n$ is isomorphic to $P$ for $n\ge1$, so
the assumption of the corollary is satisfied.
 The resulting ring
$S_{\ref{projective-module-resolution-construction}}=
\Hom_R(P,P)$ is the ring of infinite, column-finite $A\times A$
matrices with entries from the commutative polynomial ring
$R=k[x_\alpha:\alpha\in A]$ in infinitely many variables.

 In the context of
Corollaries~\ref{first-contramodule-construction-cor}
and~\ref{second-contramodule-construction-cor}, it makes sense to
introduce the notation $\J_{\ref{first-contramodule-construction-cor}}$
for the cofree comodule $\J=\C\ot_kW$ appearing in
Corollary~\ref{first-contramodule-construction-cor} and
the notation $\fP_{\ref{first-contramodule-construction-cor}}$ for
the free contramodule $\fP=\Hom_k(\C,W)$ mentioned in the discussion
in its proof.
 Then the notation $\fP_{\ref{second-contramodule-construction-cor}}$
can be used for the bigger free contramodule $\fP=\Hom_k(\C,W^*)$
from Corollary~\ref{second-contramodule-construction-cor}, and we
can also denote by $\J_{\ref{second-contramodule-construction-cor}}$
the corresponding cofree comodule $\J=\C\ot_kW^*$.

 The ring $S_{\ref{first-contramodule-construction-cor}}=
\Hom_\C(\J_{\ref{first-contramodule-construction-cor}},
\J_{\ref{first-contramodule-construction-cor}})^\rop=
\Hom^\C(\fP_{\ref{first-contramodule-construction-cor}},
\fP_{\ref{first-contramodule-construction-cor}})^\rop$ is the ring
of infinite, row-zero-convergent $A\times A$ matrices with entries
from the topological commutative formal power series ring
$\R=\C^*=\varprojlim_{B\subset A}k[[x_\alpha:\alpha\in B]]$ in
infinitely many variables.
 Such rings of row-zero-convergent matrices were discussed in
the papers~\cite[Example~7.10]{PS1} and~\cite[Section~5]{PS3}.

 Let $D$ denote the indexing set of a basis $\{y_\delta\:\delta\in D\}$
in the $k$\+vector space~$W^*$.
 The cardinality $|D|$ of the set $D$ is equal to $|k|^{|A|}$, where
$|k|$~is the cardinality of the field~$k$ and $|A|$ is the cardinality
of the set~$A$.
 Then the ring $S_{\ref{second-contramodule-construction-cor}}=
\Hom_\C(\J_{\ref{second-contramodule-construction-cor}},
\J_{\ref{second-contramodule-construction-cor}})=
\Hom^\C(\fP_{\ref{second-contramodule-construction-cor}},
\fP_{\ref{second-contramodule-construction-cor}})$ is the ring
of infinite, column-zero-convergent $D\times D$ matrices with entries
from the topological commutative formal power series ring $\R$
in infinitely many variables indexed by~$A$.

 Endowed with its natural topology, the ring
$S_{\ref{first-contramodule-construction-cor}}$ becomes a complete,
separated right linear topological ring (i.~e., a topological ring
with a base of neighborhoods of zero formed by open right ideals).
 Such topological rings were discussed in
the papers~\cite{Pproperf,PS1,PS3}.
 Moreover, the ring $S_{\ref{first-contramodule-construction-cor}}$
is \emph{topologically left perfect} in the sense
of~\cite[Section~14]{PS3} (as one can see from the discussion
in~\cite[Section~7.3]{PS1} or~\cite[Section~5]{PS3} together
with~\cite[Example~12.3]{Pproperf}).
 Similarly, the ring $S_{\ref{second-contramodule-construction-cor}}$
is a complete, separated, left linear topological ring which is,
moreover, topologically right perfect.
 In other words, the ring
$S_{\ref{second-contramodule-construction-cor}}$ is the endomorphism
ring of a module with perfect decomposition~\cite[Section~10]{PS3},
while $S_{\ref{first-contramodule-construction-cor}}$ is the opposite
ring to the endomorphism ring of such a module.

 Viewed as abstract rings, both the rings
$S_{\ref{first-contramodule-construction-cor}}$ and
$S_{\ref{second-contramodule-construction-cor}}$ are
\emph{semiregular} in the sense of~\cite[Section~4]{Ang};
see the discussion in~\cite[Remark~10.5]{PS3}.
 (The semiregularity is a left-right symmetric property.)
 The ring $S_{\ref{projective-module-resolution-construction}}$, on
the other hand, has vanishing Jacobson radical.

 The ring $S_{\ref{second-contramodule-construction-cor}}$ differs
from the opposite ring to
$S_{\ref{first-contramodule-construction-cor}}$ for the only reason
that the cardinality of the set $D$ is larger than that of
the set~$A$.
 One can employ the bigger cofree comodule
$\J_{\ref{second-contramodule-construction-cor}}=\C\ot_k W^*$
in lieu of the cofree comodule
$\J_{\ref{first-contramodule-construction-cor}}=\C\ot_k W$ in
the construction of
Corollary~\ref{first-contramodule-construction-cor} (while leaving
the rest of the construction unchanged).
 This will produce a pair of opposite rings $S$ and $S^\rop$, both
of them semiregular, both of them satisfying all the claims of
Conclusion~\ref{main-conclusion}.

\begin{rem}
 \emph{None} of the rings
$S_{\ref{projective-module-resolution-construction}}$,
$S_{\ref{first-contramodule-construction-cor}}$,
and $S_{\ref{second-contramodule-construction-cor}}$ is Noetherian
on either side.
 Indeed, consider the ring $T=\Hom_k(k^{(E)},k^{(E)})$ of infinite,
column-finite $E\times E$ matrices with entries from the field~$k$
(where $E$ is a set and $k^{(E)}$ is the $k$\+vector space with
a basis indexed by~$E$).
 Then $T$ is a quotient ring of
$S_{\ref{projective-module-resolution-construction}}$ (for $E=A$)
and of $S_{\ref{second-contramodule-construction-cor}}$ (for $E=D$),
while the opposite ring to $T$ is a quotient ring of
$S_{\ref{first-contramodule-construction-cor}}$ (for $E=A$).
 The ring $T$ is von Neumann regular, so it cannot be left or
right Noetherian (as any one-sided Noetherian von Neumann regular
ring is semisimple Artinian).
\end{rem}

\begin{rem}
 A simpler construction of rings and complexes of modules satisfying
Conclusion~\ref{main-conclusion} than the one discussed above exists
(see Example~\ref{universal-example} below).
 But the following na\"\i ve attempt at constructing an example
for Conclusion~\ref{main-conclusion}(b) fails.

 The commutative ring $R=k[x_\alpha:\alpha\in A]$ of polynomials
in infinitely many variables over a field~$k$ is not Noetherian,
but it is coherent.
 Hence the class of flat $R$\+modules is closed under infinite
products in $R\Modl$, and it follows that the $R$\+module
$\Hom_R(Q,P)$ is flat for any projective $R$\+module $Q$ and
flat $R$\+module~$P$.
 Thus the complex~\eqref{acyclic-complex-of-projs-from-proj-res}
from Corollary~\ref{projective-module-resolution-construction} is
a bounded below, acyclic complex of flat $R$\+modules.
 One does not even need the $R$\+modules $P_n$ to be direct summands
of $P$ for this claim to hold; it suffices to take $P=R$.

 However, this is \emph{not} an example for
Conclusion~\ref{main-conclusion}(b), because the complex of
$R$\+modules~\eqref{acyclic-complex-of-projs-from-proj-res} is
actually pure acyclic (for any flat $R$\+module~$P$).
 Indeed, it suffices to show that, for any finitely presented
$R$\+module $M$, applying the functor $M\ot_R{-}$ preserves acyclicity
of the complex~\eqref{acyclic-complex-of-projs-from-proj-res}.
 Denote the complex~\eqref{acyclic-complex-of-projs-from-proj-res}
by~$F^\bu$.

 Any finitely presented module over the ring of polynomials $R$ in
infinitely many variables has a finite projective resolution $G_\bu$
by finitely generated projective $R$\+modules.
 Since $F^\bu$ is a complex of flat $R$\+modules and $G_\bu$ is
a finite resolution, the complexes $M\ot_RF^\bu$ and $G_\bu\ot_RF^\bu$
are quasi-isomorphic.
 Finally, viewed as an object of the homotopy category of complexes
of $R$\+modules $\sK(R\Modl)$, the complex $G_\bu\ot_RF^\bu$ belongs
to the thick subcategory spanned by the complex~$F^\bu$ (since
the complex $G_\bu$ belongs to the thick subcategory spanned by
the one-term complex of $R$\+modules~$R$).
 As the complex $F^\bu$ is acyclic, so is the complex $G_\bu\ot_RF^\bu$.
\end{rem}

\begin{ex} \label{universal-example}
 The following example has a different nature than all the previous
examples in this paper.
 It was communicated to the author by A.~Canonaco and is reproduced here
with his kind permission.

 Suppose that we have a bounded below complex of \emph{free modules
with one generator} over a ring~$S$.
 Obviously, such a complex of (left) modules has the form
\begin{equation} \label{one-generator-free-modules-complex}
 0\lrarrow S\overset{*z_0}\lrarrow S\overset{*z_1}\lrarrow
 S\overset{*z_2}\lrarrow S\lrarrow\dotsb,
\end{equation}
where $z_0$, $z_1$, $z_2$,~\dots\ is some sequence of elements in~$S$.
 For the sequence of maps~\eqref{one-generator-free-modules-complex}
to be a complex, the equation $z_nz_{n+1}=0$ has to be satisfied in
$S$ for all integers $n\ge0$.

 Now let $k$~be a field and $S_{\mathrm{uni}}$ be the $k$\+algebra
generated by a sequence of elements $x_0$, $x_1$, $x_2$,~\dots\
with the imposed relations $x_nx_{n+1}=0$ for all $n\ge0$, and no
other relations.
 Then one can easily see that any element $s\in S_{\mathrm{uni}}$
satisfying the equation $sx_0=0$ vanishes, while any element
$s\in S_{\mathrm{uni}}$ satisfying $sx_{n+1}=0$ with $n\ge0$ has
the form $s=tx_n$ for some $t\in S_{\mathrm{uni}}$.
 It suffices to represent $s$~as a $k$\+linear combination of monomials
in the variables~$x_n$, \,$n\ge0$, etc.
 In other words, the bounded below complex of free
$S_{\mathrm{uni}}$\+modules
\begin{equation} \label{universal-example-eqn}
 0\lrarrow S_{\mathrm{uni}}\overset{*x_0}\lrarrow
 S_{\mathrm{uni}}\overset{*x_1}\lrarrow
 S_{\mathrm{uni}}\overset{*x_2}\lrarrow S_{\mathrm{uni}}\lrarrow\dotsb
\end{equation}
is acyclic.
 On the other hand, if~$k$ is endowed with the right
$S_{\mathrm{uni}}$\+module structure in which all the elements~$x_n$
act by zero in~$k$, then applying the functor
$k\ot_{S_{\mathrm{uni}}}{-}$ to
the complex~\eqref{universal-example-eqn} produces a nonacyclic complex
with zero differential.
 So the complex~\eqref{universal-example-eqn} is \emph{not}
contractible.

 The bounded below complex of free $S_{\mathrm{uni}}$\+modules with one
generator~\eqref{universal-example-eqn} is \emph{universal} in
the following sense.
 Let $S$ be an associative $k$\+algebra and $C^\bu$ be a complex of
free $S$\+modules with one generator such that $C^i=0$ for $i<0$.
 Then there exists a $k$\+algebra homomorphism $f\:S_{\mathrm{uni}}
\rarrow S$ such that the complex $C^\bu$ is obtained by applying
the functor of extension of scalars $S\ot_{S_{\mathrm{uni}}}{-}$ to
the complex~\eqref{universal-example-eqn}.
 Indeed, the complex $C^\bu$ has
the form~\eqref{one-generator-free-modules-complex} for some
elements $z_n\in S$, \,$n\ge0$ satisfying the equations $z_nz_{n+1}=0$,
and it remains to let $f\:S_{\mathrm{uni}}\rarrow S$ be
the homomorphism taking~$x_n$ to~$z_n$ for every $n\ge0$.
\end{ex}

 While the example in Example~\ref{universal-example} is certainly
simpler (to construct and prove its properties) than the examples
in Corollaries~\ref{injective-module-coresolution-construction},
\ref{projective-module-resolution-construction},
\ref{first-contramodule-construction-cor},
and~\ref{second-contramodule-construction-cor}, \emph{no} example of
a bounded below, noncontractible, acyclic complex of projective modules
(or of a bounded above, noncontractible, acyclic complex of injective
modules) can be \emph{too} simple.
 The results of~\cite[Appendix~A]{Sha} demonstrate this.


\bigskip
\begin{thebibliography}{99}
\smallskip

\bibitem{Ang}
 L.~Angeleri H\"ugel.
   Covers and envelopes via endoproperties of modules.
\textit{Proc.\ London Math.\ Soc.}\ \textbf{86}, \#3, p.~649--665, 2003.

\bibitem{CaE}
 H.~Cartan, S.~Eilenberg.
   Homological algebra.
Princeton Landmarks in Mathematics and Physics,
Princeton, 1956--1999.

\bibitem{Dr}
 A.~Dress.
   On the decomposition of modules.
\textit{Bull.\ Amer.\ Math.\ Soc.}\ \textbf{75}, p.~984--986, 1969.

\bibitem{Jen}
 C.~U.~Jensen.
   Les foncteurs d\'eriv\'es de $\varprojlim$ et leurs applications
en th\'eorie des modules.
\textit{Lecture Notes in Math.} \textbf{254}, Springer, 1972.

\bibitem{Neem}
 A.~Neeman.
   The homotopy category of flat modules, and Grothendieck duality.
\textit{Inventiones Math.}\ \textbf{174}, \#2, p.~225--308, 2008.

\bibitem{Psemi}
 L.~Positselski.
   Homological algebra of semimodules and semicontramodules:
Semi-infinite homological algebra of associative algebraic structures.
 Appendix~C in collaboration with D.~Rumynin; Appendix~D in
collaboration with S.~Arkhipov.
 Monografie Matematyczne vol.~70, Birkh\"auser/Springer Basel, 2010. 
xxiv+349~pp. \texttt{arXiv:0708.3398 [math.CT]}

\bibitem{Prev}
 L.~Positselski.
   Contramodules.
\textit{Confluentes Math.}\ \textbf{13}, \#2, p.~93--182, 2021.
\texttt{arXiv:1503.00991 [math.CT]}

\bibitem{Pproperf}
 L.~Positselski.
   Contramodules over pro-perfect topological rings.
\textit{Forum Mathematicum} \textbf{34}, \#1, p.~1--39, 2022.
\texttt{arXiv:1807.10671 [math.CT]}

\bibitem{Prel}
 L.~Positselski.
   Relative nonhomogeneous Koszul duality.
Frontiers in Mathematics, Birkh\"auser/Springer Nature, Cham,
Switzerland, 2021.  xxix+278~pp.  \texttt{arXiv:1911.07402 [math.RA]}
{\hbadness=5200\par}

\bibitem{Pksurv}
 L.~Positselski.
   Differential graded Koszul duality: An introductory survey.
\textit{Bulletin of the London Math.\ Society} \textbf{55}, \#4,
p.~1551--1640, 2023.  \texttt{arXiv:2207.07063 [math.CT]}

\bibitem{PS1}
 L.~Positselski, J.~\v St\!'ov\'\i\v cek.
   The tilting-cotilting correspondence.
\textit{Internat.\ Math.\ Research Notices} \textbf{2021}, \#1,
p.~189--274, 2021.  \texttt{arXiv:1710.02230 [math.CT]} 

\bibitem{PS3}
 L.~Positselski, J.~\v St\!'ov\'\i\v cek.
   Topologically semisimple and topologically perfect topological rings.
\textit{Publicacions Matem\`atiques} \textbf{66}, \#2, p.~457--540,
2022.  \texttt{arXiv:1909.12203 [math.CT]}

\bibitem{Pros}
 F.~Prosmans.
   Derived limits in quasi-abelian categories.
\textit{Bulletin de la Soc.\ Royale des Sci.\ de Li\`ege}
\textbf{68}, \#5--6, p.~335--401, 1999.

\bibitem{Rick}
 J.~Rickard.
   Unbounded derived categories and the finitistic dimension conjecture.
\textit{Advances in Math.}\ \textbf{354}, article~ID~106735,
21~pp., 2019.  \texttt{arXiv:1804.09801 [math.RT]}

\bibitem{Sha0}
 L.~Shaul.
   The finitistic dimension conjecture via DG\+rings.
Electronic preprint \texttt{arXiv:2209.02068 [math.RA]}.
{\hbadness=10000\par}

\bibitem{Sha}
 L.~Shaul.
   Acyclic complexes of injectives and finitistic dimensions.
With an appendix by T.~Nakamura and P.~Thompson.
Electronic preprint \texttt{arXiv:2303.08756 [math.RA]}.

\bibitem{Swe}
 M.~E.~Sweedler.
   Hopf algebras.
Mathematics Lecture Note Series, W.~A.~Benjamin, Inc., New York, 1969.

\end{thebibliography}
\end{document}